\documentclass[12pt,a4paper]{article}
\usepackage[T1]{fontenc}
\usepackage{lmodern}
\usepackage[utf8]{inputenc}
\usepackage{amsthm}
\usepackage{amsmath,amssymb}
\usepackage{graphics}
\usepackage{todonotes}
\usepackage{lscape,graphicx}
\usepackage{enumitem}
\usepackage{amsfonts}
\usepackage[T1]{fontenc}
\usepackage{longtable}
\usepackage{authblk}
\usepackage{lipsum}
\usepackage{epsfig}
\usepackage{float}
\usepackage{hyperref}
\usepackage{comment}
\usepackage[normalem]{ulem}
\usepackage{multicol}
\usepackage{fancyhdr}
\usepackage[english]{babel}
\usepackage{mathrsfs}
\usepackage{tocloft}
\usepackage{xcolor}
\usepackage{titletoc}
\usepackage{mathtools}
\usepackage[toc,page]{appendix}

\addto{\captionsenglish}{}
\makeatletter
\def\thm@space@setup{%
  \thm@preskip=1cm plus .5cm minus .5cm
  \thm@postskip=.5cm plus .6cm minus .5cm 
}

\makeatother

\newtheorem{thm}{Theorem}

\newtheorem{lma}{Lemma}

\newtheorem{cor}{Corollary}

\newtheorem{rmk}{Remark}
\newtheorem{ex}{Example}
\numberwithin{thm}{section}
\numberwithin{lma}{section}
\numberwithin{dfn}{section}
\numberwithin{cor}{section}
\numberwithin{rmk}{section}
\numberwithin{prop}{section}

\newcommand*{\thmref}[1]{Theorem~\ref{#1}}

\newcommand*{\lmaref}[1]{Lemma~\ref{#1}}

\newcommand*{\corref}[1]{Corollary~\ref{#1}}

\title{On limiting distributions of arithmetic functions}

\date{}

\begin{document}

\author{Sourabhashis Das}

\newcommand{\Addresses}{{
  \bigskip
  \footnotesize

  Sourabhashis Das (Corresponding author), Department of Pure Mathematics, University of Waterloo, 200 University Avenue West, Waterloo, Ontario, Canada, N2L 3G1. \\
  Email address: \texttt{s57das@uwaterloo.ca}
}}

%
%
%

\maketitle 


\begin{abstract}
For a natural number $n$, let $M(n)$ denote the maximum exponent of any prime power dividing $n$, and let $m(n)$ denote the minimum exponent of any prime power dividing $n$. We study the second moments of these arithmetic functions and establish their limiting distributions. We introduce a new discrete probabilistic distribution dependent on a function $f$ taking values in $[0,1]$, study its first two moments, and provide examples of several arithmetic functions satisfying such distribution as their limiting behavior.  
\end{abstract}

\section{Introduction}

The study of distributions of arithmetic functions offers valuable insights into the divisibility of integers, the behavior of prime numbers, and other foundational structures in mathematics, making it one of the core problems in number theory. In this article, we examine the distributions of two important arithmetic functions.

For a natural number $n$, let its prime factorization be given by
\begin{equation}\label{factorization}
    n = p_1^{\alpha_1} p_2^{\alpha_2} \cdots p_r^{\alpha_r},
\end{equation}
where $p_i'$s are prime numbers and $\alpha_i'$s are natural numbers. Let $M(n)$ and $m(n)$ respectively denote the maximum and the minimum exponent of any prime power dividing $n$. Here, for $n > 1$,
$$M(n) = \max \{ \alpha_1, \alpha_2, \cdots, \alpha_r \},$$
and
$$m(n) = \min \{ \alpha_1, \alpha_2, \cdots, \alpha_r \}.$$
For convenience, we define $M(1) = m(1) = 1$. 
Observe that these arithmetic functions are neither additive nor multiplicative, yet they have been extensively studied in the literature (see \cite{cao, cd, niven, sinha, susi}). The average distributions of $m(n)$ and $M(n)$ are analyzed in \cite{cao, cd}, while \cite{niven} further establishes their normal orders. Studies in \cite{sinha, susi} refine the understanding of the error terms in these average distributions. In this article, we determine their second moments and examine the limiting probabilistic distributions they obey. Notably, we connect the asymptotic behavior of $M(n)$ to patterns observed in other well-known arithmetic functions, offering a unified framework for their study.

Let $B_1$ be the constant defined as 
$$B_1 := 1 + \sum_{k=2}^\infty \left( 1 - \frac{1}{\zeta(k)} \right) \approx 1.705211140.$$
The average distributions of $M(n)$ is given by (see \cite[Theorem 3.1]{susi}):
\begin{equation}\label{average-M(n)}
    \sum_{n \leq x} M(n) = B_1 x + O(x^{1/2}).
\end{equation}

For a natural number $k$, let us define the constants
\begin{equation}\label{gamma0k}
    \gamma_{0,2} = \frac{\zeta(3/2)}{\zeta(3)}, \quad \gamma_{0,k} = \prod_p \left( 1 + \sum_{m=k+1}^{2k-1} p^{-m/k} \right),
\end{equation}
\begin{equation}\label{gamma1k}
    \gamma_{1,2} = \frac{\zeta(2/3)}{\zeta(2)}, \quad \text{and} 
    \quad \gamma_{1,k} = \zeta \left( \frac{k}{k+1}\right) \prod_p \left( 1 + \sum_{m=k+2}^{2k-1} p^{-m/(k+1)} - \sum_{m=2k+2}^{3k} p^{-m/(k+1)} \right),
\end{equation}
where the products run over all primes, and where $\zeta(s)$ denote the classical Riemann $\zeta$-function. The average distributions of $m(n)$ is given by (see \cite[Theorem 3.3]{susi}):
\begin{equation}\label{average-m(n)}
    \sum_{n \leq x} m(n) = x + \gamma_{0,2} x^{1/2} + (\gamma_{0,3} + \gamma_{1,2}) x^{1/3} + O(x^{1/4}).
\end{equation}

In this work, we study the second moments of $m(n)$ and $M(n)$. For $m(n)$, we prove:

\begin{thm}\label{m(n)-secmom}
Let $x > 2$. We have
$$\sum_{n \leq x} m^2(n) = x + 3 \gamma_{0,2} x^{1/2} + (3 \gamma_{1,2} + 5 \gamma_{0,3}) x^{1/3} + O(x^{1/4}),$$
where $\gamma_{0,k}$ and $\gamma_{1,k}$ are defined in \eqref{gamma0k} and \eqref{gamma1k} respectively.
\end{thm}

As a result, we prove the variance of $m(n)$ as:

\begin{cor}\label{var-m(n)}
We have
$$\frac{1}{x} \sum_{n \leq x} (m(n) - 1)^2 = \frac{\gamma_{0,2}}{x^{1/2}} + \frac{3 \gamma_{0,3} + \gamma_{1,2}}{x^{2/3}} + O \left( \frac{1}{x^{3/4}} \right).$$
\end{cor}

For $M(n)$, we prove:

\begin{thm}\label{M(n)-secmom}
Let $x > 2$. We have
$$\sum_{n \leq x} M^2(n) = B_2 x + O(x^{1/2}),$$
where
$$B_2 = B_1 + 2 \sum_{k=2}^{\infty} (k - 1) \left( 1 - \frac{1}{\zeta(k)} \right) \approx 4.301302400.$$
\end{thm}

We prove the variance of $M(n)$ as:

\begin{cor}\label{var-M(n)}
We have
$$\frac{1}{x} \sum_{n \leq x} (M(n) - B_1)^2 = B_2 - B_1^2 + O(x^{-1/2}),$$
with $B_2 - B_1^2 \approx 1.393557368.$
\end{cor}

Next, we introduce the definition of the normal order of an arithmetic function. Let $f, F : \mathbb{N} \cup \{ 0\} \rightarrow \mathbb{R}_{\geq 0}$ be two functions such that $F$ is non-decreasing. Then, $f(n)$ is said to have normal order $F(n)$ if for any $\epsilon > 0$, the number of $n \leq x$ that do not satisfy the inequality
$$(1-\epsilon) F(n) \leq f(n) \leq (1+ \epsilon) F(n)$$
is $o(x)$ as $x \rightarrow \infty$. 

Niven \cite{niven} established that $m(n)$ has normal order 1 and $M(n)$ has no normal order. One can verify this result by noticing that the numbers $n$ with $m(n) > 1$ would be square-full and the set of all square-full numbers has a density of 0. Similarly, for $M(n)$, all square-free numbers take value 1, and all non-square-free numbers take values greater than $1$. Since both these subsets have positive density and the average order of $M(n)$ lies between 1 and 2, thus $M(n)$ can not have a normal order.

In this article, we establish the limiting distributions of arithmetic functions $m(n)$ and $M(n)$. First, we define the limiting distribution of an arithmetic function.

Let $x > 2$ and $k \in \mathbb{N} \cup \{0\}$. We define the probability of an arithmetic function $A(n)$ on $[0,x]$ taking value $k$ as
$$P_x(A = k) := \frac{1}{x} \{ n \leq x : A(n) = k \}.$$
Let the limiting distribution of $A$ be defined as the probabilistic distribution (if exists)
$$P_\infty(A = k) := \lim_{x \rightarrow \infty} P_x(A = k).$$
Note that degenerate distribution, denoted as $A_0$, has the probability distribution
\begin{equation}\label{PD}
    P(A_0 = k) = \begin{cases}
    1 & \text{ if } k = 1, \\
    0 & \text{ otherwise}.
\end{cases}
\end{equation}

For the limiting distribution of $m(n)$, we prove:

\begin{thm}\label{lim-dis-m(n)}
   The limiting distribution of $m(n)$ is the degenerate distribution.
\end{thm}

Arithmetic functions like $\omega(n)$ which counts the number of distinct prime factors of a natural number $n$, have normal order $\log \log n$ and satisfy a limiting normal distribution. This result was first proved by Erd\H{o}s and Kac \cite{ErdosKac} and is termed the Erd\H{o}s-Kac theorem in the literature. The result shows that while $m(n)$ possesses a normal order, it does not follow a limiting normal distribution. This demonstrates that possessing a normal order does not guarantee that an arithmetic function will follow a limiting normal distribution.


We now examine the limiting distribution of $M(n)$, considering $M(n)$ as a random variable with values in $\mathbb{N}$. Using \eqref{average-M(n)}, \corref{var-M(n)}, and a later result showing that the limiting distribution of $M(n)$ assigns decreasing probabilities to values $k$, with the highest probability at $k=1$ (equal to $1/\zeta(2)$), we deduce that this distribution has a mean of $B_1 \approx 1.705211140$, variance $B_2 - B_1^2 \approx 1.393557368$, and mode $1/\zeta(2) = 0.60792710185$. This non-uniform distribution implies distinct probabilities for different values, and since the mean differs from the mode, $M(n)$ cannot exhibit a normal distribution. Furthermore, the unequal mean and variance indicate that $M(n)$ does not follow a Poisson distribution. Additionally, this distribution does not align with any well-known probabilistic distributions in the literature, prompting us to explore a new distribution that accurately describes this limiting behavior.

To analyze the limiting distribution of $M(n)$, we introduce a new discrete probability distribution defined by a function $f$ over non-negative integers. This distribution shares a key characteristic with Benford's law: the probability of a random variable taking the value $n$ decreases as $n$ increases. For arithmetic functions, this means the function tends to take lower values more frequently. Unlike Benford's distribution, which is typically applied to finite sample spaces, our framework applies to a countable sample space. This new distribution enables us to capture the limiting behavior of several well-known arithmetic functions under a unified approach.

Let $f : \mathbb{N} \cup \{0\} \rightarrow [0,1]$ be a non-decreasing function, i.e., $f(m) \geq f(n)$ for all $m > n$, with $f(0) = 0$. Let $f$ satisfy the following two properties:
\begin{enumerate}
    \item[(A)] $\lim_{n \rightarrow \infty} f(n) = 1$, and
    \item[(B)] The tail, $1 - f(n)$ decreases at the following rate:
    $$1 - f(n) \ll \frac{1}{n^{2+\epsilon}} \quad \text{ as } n \rightarrow \infty,$$
    for some small $\epsilon > 0$.
\end{enumerate}

Let $N_0$ be a fixed natural number and the function $f_0$ be defined as
\begin{equation}\label{f0n}
f_0(n) := \begin{cases}
    0 & \text{ if } n = 0 \\
    \frac{\log(n)}{\log N_0} & \text{ if } n = 1, \cdots, N_0-1 \\
    1 & \text{ otherwise.}
\end{cases}
\end{equation}
It easily follows that $f_0$ satisfies Properties (A) and (B) above. Another example of such functions is
\begin{equation}\label{f1n}
    f_1(n) := \begin{cases}
    0 & \text{ if } n = 0,1, \\
    1/\zeta(n) & \text{ otherwise.}
\end{cases}
\end{equation}
In fact, we prove that $f_1(n)$ satisfies Properties (A) and (B) in \lmaref{arith-zeta}.

For a non-decreasing function $f$ defined above, we define a discrete random variable $X_f$ on $\mathbb{N} \cup \{ 0 \}$ as $X_f$ taking value $k$ with probability $f(k+1) - f(k)$. We would call $X_f$ an \textit{arithmetic-$f$ random variable} and its corresponding distribution an \textit{arithmetic-$f$ distribution}.

\begin{rmk}
Notice that, if $N_0 \in \mathbb{N}$ and we restrict ourselves to the finite space $\{0,1,\cdots, N_0-1\}$, then $X_{f_0}$ with $f_0$ defined in \eqref{f0n} is indeed a random variable that satisfies the classical Benford's law.    
\end{rmk}

For such an $X_f$ and for an integer $r \geq 1$, let $\mu_{f,r}$ denote the $r$-th moment of $X_f$, and given as (if the limit exists)
$$\mu_{f,r} := \lim_{n \rightarrow \infty} \sum_{k=0}^n k^r (f(k+1) - f(k)).$$

Note that Property (A) is necessary to ensure that probabilities of $X_f$ sum up to $1$, and Property (B) ensures that the first two moments of $X_f$ exist and are finite. We demonstrate these points in detail in Section \ref{moment}. Property (B) can be adjusted by further restricting the growth of $1-f(n)$ to secure the finiteness of higher moments. However, in this article, we limit our focus to the first two moments, with plans to address higher moments in future work.

Let $D_1$ be a constant defined as
$$D_1 := \sum_{k=2}^{\infty} (1 - f(k)),$$
where the convergence of the sum follows from Property (B). Let $D_2$ be another constant defined as
$$D_2 := D_1 + 2 \sum_{k=2}^{\infty} (k -1) (1 - f(k)).$$

We establish the first and the second moment of an arithmetic distribution as:
\begin{thm}\label{thm-dis}
    Let $X_f$ be a random variable that satisfies an arithmetic-$f$ distribution, where $f$ is non-decreasing on $\mathbb{N} \cup \{0\}$ with $f(0) = 0$ and obeys Properties (A) and (B). Then
    $$\mu_{f,1} = 1 - f(1) + D_1,$$
    and
    $$\mu_{f,2} = 1 - f(1) + D_2.$$
\end{thm}

As an application of the above theorem, we prove that $M(n)$ has an arithmetic distribution as its limiting distribution. In particular, we prove:

\begin{thm}\label{lim-dis-M(n)}
 $M(n)$ satisfies the arithmetic-$f_1$ distribution, where $f_1$ is defined in \eqref{f1n}, as it limiting distribution.  
\end{thm}

Finally, Section 5.3 provides additional examples of arithmetic functions satisfying arithmetic-$f$ distributions for distinct $f$ as their limiting behavior.

\section{The second moment of \texorpdfstring{$m(n)$}{}}

Let $n$ be a natural number with a factorization given in \eqref{factorization}:
$$n = p_1^{\alpha_1} p_2^{\alpha_2} \cdots p_r^{\alpha_r}.$$
Let $k \geq 2$ be an integer. We say $n$ is \textit{$k$-free} if $\alpha_i \leq k-1$ for all $i \in \{1, \ldots, r \}$, and we say $n$ is \textit{$k$-full} if $\alpha_i \geq k$ for all $i \in \{1, \ldots, r \}$. Let $S_k$ be the set of all $k$-free numbers, and let $N_k$ be the set of all $k$-full numbers. For convenience, $1$ is both $k$-free and $k$-full.

For establishing the second moment of $m(n)$, we need the following density result for $k$-full numbers:

\begin{lma}\label{hfullintegers}(\cite[Theorem 4]{cd} or \cite[Lemma 2.3]{susi})
    Let $x > 2$. Let $N_k(x)$ denote the number of $k$-full integers less than or equal to $x$. Then
    $$N_k(x) = \gamma_{0,k} x^{1/k} + \gamma_{1,k} x^{1/(k+1)} + O(x^{1/(k+2)}),$$
    where $\gamma_{0,k}$ and $\gamma_{1,k}$ are defined in \eqref{gamma0k} and \eqref{gamma1k} respectively.
\end{lma}

Using the above lemma, we prove the following result:

\begin{proof}[\textbf{Proof of \thmref{m(n)-secmom}}]
    Note that for any natural number $n$ satisfying $n \leq x$, $m(n) \leq \left\lceil \frac{\log x}{\log 2}\right\rceil := j$. Moreover, for any integer $k \geq 2$, there are $N_{k-1}(x) - N_k(x)$ integers $n$ less than or equal to $x$ with $m(n) = k-1$. Thus, using $m(1) =1$, $N_{j+1}(x) = 1$ since $1$ is $k$-full for any $k \geq 2$ by convention, and $N_1(x) = x + O(1)$, we have
    \begin{align*}
        \sum_{n \leq x} m^2(n) & = 1 + \sum_{k=2}^{j+1} (k-1)^2 (N_{k-1}(x) - N_k(x)) \\
        & = 1 + \sum_{k=2}^{j+1} (k-1)^2 N_{k-1}(x) - \sum_{k=2}^{j+1} (k-1)^2 N_{k}(x) \\
        & = 1 - j^2 + N_1(x) + \sum_{k=2}^{j} k^2 N_k(x) - \sum_{k=2}^j (k-1)^2 N_k(x) \\
        & = 1 - j^2 + x + \sum_{k=2}^j (2k-1) N_k(x) + O(1) \\
        & = x + 3 N_2(x) + 5 N_3(x) + \sum_{k=4}^j (2k-1) N_k(x) + O((\log x)^2).
    \end{align*}
    By \lmaref{hfullintegers}, we have
    $$N_2(x) = \gamma_{0,2} x^{1/2} + \gamma_{1,2} x^{1/3} + O(x^{1/4}),$$
    $$N_3(x) = \gamma_{0,3} x^{1/3} + O(x^{1/4}),$$
    and
    $$\sum_{k=4}^j (2k-1) N_k(x) \ll 7 x^{1/4} + x^{1/5} \sum_{k=5}^j (2k-1) \ll x^{1/4}.$$
    As a result, we obtain
    $$\sum_{n \leq x} m^2(n) = x + 3 \gamma_{0,2} x^{1/2} + (3 \gamma_{1,2} + 5 \gamma_{0,3}) x^{1/3} + O(x^{1/4}),$$
    which completes the proof.
\end{proof}
\begin{rmk}
    The error term in the above theorem can be improved by using the result of Ivi\'c and Shiu \cite[Theorem 1]{is}.
\end{rmk}

As a corollary, we prove the variance of $m(n)$ as:

\begin{proof}[\textbf{Proof of \corref{var-m(n)}}]
    By \eqref{average-m(n)}, \thmref{m(n)-secmom} and the fact that $\sum_{n \leq x} 1 = x + O(1)$, we have
    \begin{align*}
        \sum_{n \leq x} (m(n) - 1)^2 & =  \sum_{n \leq x} m^2(n) - 2 \sum_{n \leq x} m(n) + \sum_{n \leq x} 1 \\
        & = x + 3 \gamma_{0,2} x^{1/2} + (3 \gamma_{1,2} + 5 \gamma_{0,3}) x^{1/3} + O(x^{1/4}) \\
        & \hspace{.5cm} - 2 (x + \gamma_{0,2} x^{1/2} + (\gamma_{0,3} + \gamma_{1,2}) x^{1/3}) + O(x^{1/4}) \\
        & \hspace{1cm} + x + O(1) \\
        & = \gamma_{0,2} x^{1/2} + (3 \gamma_{0,3} + \gamma_{1,2}) x^{1/3} + O(x^{1/4}).
    \end{align*}
Dividing both sides above by $x$ completes the proof.
\end{proof}

\section{The second moment for \texorpdfstring{$M(n)$}{}}

In this section, we prove the second moment of $M(n)$ over natural numbers as the following:

\begin{proof}[\textbf{Proof of \thmref{M(n)-secmom}}]
Note that for any $n$ satisfying $n \leq x$, $M(n) \leq \left\lceil \frac{\log x}{\log 2}\right\rceil := j$. For any integer $k \geq 2$, let $S_k(x)$ denote the number of $k$-free integers less than or equal to $x$. Then there are $S_k(x) - S_{k-1}(x)$ integers $n$ less than or equal to $x$ with $M(n) = k-1$. Moreover, if $1 < k \leq j$, it is well-known that (see \cite[(4)]{jala})
\begin{equation}\label{distribution-skx}
   S_k(x) = \frac{x}{\zeta(k)} + O \big( x^\frac{1}{k} \big),
\end{equation}
where the implied constant depends on $k$. Thus, using $S_{j+1}(x) = x + O(1)$ and the convention $S_1(x) = 1$, we have
\begin{align*}
    \sum_{n \leq x} M^2(n) & = \sum_{k=2}^{j+1} (k-1)^2 (S_k(x) - S_{k-1}(x)) \\
    & = j^2 x + \sum_{k=2}^{j} (k-1)^2 S_k(x) - \sum_{k=2}^{j+1} (k-1)^2 S_{k-1}(x) \\
    & =  j^2 x + \sum_{k=2}^{j} (k-1)^2 S_k(x)- \sum_{k=2}^{j} k^2 S_{k}(x) - 1 \\
    & = j^2 x - 2 \sum_{k=2}^{j} k S_k(x) + \sum_{k=2}^{j} S_k(x) - 1 \\
    & = \left( j^2 - 2 \sum_{k=2}^{j} \frac{k}{\zeta(k)} + \sum_{k=2}^{j} \frac{1}{\zeta(k)} \right) x + O(\sum_{k=2}^{j} x^{1/k} ).
\end{align*}
Since $2 \sum_{k=2}^j k + 2 - j= j^2$, thus
\begin{align*}
    j^2 - 2 \sum_{k=2}^{j} \frac{k}{\zeta(k)} + \sum_{k=2}^{j} \frac{1}{\zeta(k)} & = 1 + 2 \sum_{k=2}^{j} k \left( 1 - \frac{1}{\zeta(k)} \right) - \sum_{k=2}^{j} \left( 1 - \frac{1}{\zeta(k)} \right) \\
    & = 1 + \sum_{k=2}^{j} (2k - 1) \left( 1 - \frac{1}{\zeta(k)} \right).
\end{align*}
Since, for any $k \geq 2$, 
$$ 1 - \frac{1}{\zeta(k)} < \frac{1}{2^{k-1}},$$
thus we can estimate the above sum as
\begin{align*}
    \sum_{k=2}^{j} (2k - 1) \left( 1 - \frac{1}{\zeta(k)} \right) & = \sum_{k=2}^{\infty} (2k - 1) \left( 1 - \frac{1}{\zeta(k)} \right) + O \left( \sum_{k=j}^\infty \frac{k}{2^{k}}\right) \\
    & = \sum_{k=2}^{\infty} (2k - 1) \left( 1 - \frac{1}{\zeta(k)} \right) + O \left( \frac{\log x}{x} \right).
\end{align*}
Combining the last three results with $\sum_{k=2}^{j} x^{1/k} \ll  x^{1/2} + x^{1/3} \log x \ll x^{1/2}$, we obtain
$$\sum_{n \leq x} M^2(n) = B_2 x + O(x^{1/2}),$$
where
$$B_2 = B_1 + 2 \sum_{k=2}^{\infty} (k - 1) \left( 1 - \frac{1}{\zeta(k)} \right).$$
This completes the proof.
\end{proof}

\begin{rmk}
   The error term in the above theorem can be improved using the result of Walfisz \cite[Satz 1, Page 129]{wa}. 
\end{rmk}

As a corollary, we prove the variance of $M(n)$ as:
\begin{proof}[\textbf{Proof of \corref{var-M(n)}}]
    By \eqref{average-M(n)}, \thmref{M(n)-secmom} and the fact that $\sum_{n \leq x} 1 = x + O(1)$, we have
    \begin{align*}
        \sum_{n \leq x} (M(n) - B_1)^2 & =  \sum_{n \leq x} M^2(n) - 2 B_1 \sum_{n \leq x} M(n) + B_1^2 \sum_{n \leq x} 1 \\
        & = \left( B_2 - 2 B_1^2 + B_1^2 \right) x + O(x^{1/2}) \\
        & = \left( B_2 - B_1^2 \right) x + O(x^{1/2}).
    \end{align*}
Dividing both sides above by $x$ completes the proof.
\end{proof}

\section{Limiting distribution of \texorpdfstring{$m(n)$}{}}

Equation \eqref{average-m(n)} shows that $m(n)$ has average order 1. Moreover, Niven in \cite[Section 3]{niven} showed that $m(n)$ has normal order 1. By \eqref{average-m(n)} and \corref{var-m(n)}, we conclude that the expectation or mean of $m(n)$ is 1 and its variance is 0. This provides evidence that $m(n)$, when considered a random variable taking values as natural numbers, satisfies the degenerate distribution. In fact, $m(n)$ satisfies the definition of an almost constant random variable. In the following, we prove this claim. 

\begin{proof}[\textbf{Proof of \thmref{lim-dis-m(n)}}]

Note that 
$$m(n) \neq 1 \Leftrightarrow n \in N_2,$$
where $N_2$ is the set of $2$-full or square-full integers. By \eqref{average-m(n)}, we have
$$P_x(m \neq 1) = \frac{|N_2(x)|}{x} \ll \frac{1}{x^{1/2}},$$
and 
$$P_x(m = 1) = \frac{x - |N_2(x)|}{x} = 1 + O \left( \frac{1}{x^{1/2}} \right).$$
Thus, as $x \rightarrow \infty$, we have
$$\lim_{x \rightarrow \infty} P_x(m \neq 1) = 0,$$
and 
$$\lim_{x \rightarrow \infty} P_x(m = 1) = 1.$$
As a result, 
$$P_\infty(m = k) = \lim_{x \rightarrow \infty} P_x (m = k) = \begin{cases}
    1 & \text{ if } m = 1, \\
    0 & \text{ otherwise}.
\end{cases}$$
Thus $P_\infty(m = k) = P(A_0 = k)$ where the probability $P(A_0 = k)$ is defined in \eqref{PD}, and hence, the limiting distribution of $m(n)$ is the degenerate distribution $A_0$.   
\end{proof}

\section{Study of arithmetic random variables and their distributions}

In this section, we first prove some distribution results for an arithmetic-$f$ random variable and then apply the results to study the limiting distribution of $M(n)$. Finally, we provide other applications of these results.

\subsection{The first and the second moment of an arithmetic-\texorpdfstring{$f$}{} random variable}\label{moment}

In this subsection, we prove the first and the second moments of the arithmetic-$f$ random variable as the following:

\begin{proof}[\textbf{Proof of \thmref{thm-dis}}]

Note that the mean or the first moment of the arithmetic-$f$ random variable is given by
$$\mu_{f,1} = \lim_{j \rightarrow \infty} \sum_{k=0}^j k (f(k+1) - f(k)).$$
We compute that
\begin{align*}
    \sum_{k=0}^j k (f(k+1) - f(k)) & = \sum_{k=1}^{j+1} (k-1) f(k) - \sum_{k=1}^j k f(k) \\
    & = \sum_{k=1}^{j+1} k f(k) - \sum_{k=1}^{j+1} f(k) - \sum_{k=1}^j k f(k) \\
    & = j f(j+1) - \sum_{k=1}^j f(k) \\
    & = \sum_{k=1}^j (f(j+1) - f(k)) \\
    & = \sum_{k=1}^{j} (f(j+1) - f(k+1)) + \sum_{k=1}^j (f(k+1) - f(k)).
\end{align*}
Next, we show that both the sums on the right side above converge, By Property (A), the second sum above goes to $1 - f(1)$ as $j \rightarrow \infty$. Moreover, the 
first sum above can be rewritten as
$$\sum_{k=1}^{j} (f(j+1) - f(k+1)) = \sum_{k=2}^{j+1} (f(j+1) - f(k)) = \sum_{k=2}^{j+1} (1 - f(k)) - \sum_{k=2}^{j+1} (1 - f(j+1)).$$
By Property (B), $\sum_{k=2}^{j+1} (1 - f(k))$ converges as $j \rightarrow \infty$. Also, by Property (B),  
$$\sum_{k=2}^{j+1} (1 - f(j+1)) \ll \sum_{k=1}^{j+1} \frac{1}{(j+1)^{2+\epsilon}} = \frac{1}{(j+1)^{1+\epsilon}}$$
which goes to 0 as $j \rightarrow \infty$. Thus, $\sum_{k=1}^{j} (f(j+1) - f(k+1))$ converges as $j \rightarrow \infty$ to the following constant
$$D_1 = \sum_{k=2}^{\infty} (1 - f(k)).$$
Combining the above results, we obtain
$$\mu_{f,1} = 1 - f(1) + D_1.$$

Note that the second moment of the arithmetic-$f$ random variable is given by
$$\mu_{f,2} = \lim_{j \rightarrow \infty} \sum_{k=0}^j k^2 (f(k+1) - f(k)).$$

We compute
\begin{align*}
    \sum_{k=0}^j k^2 (f(k+1) - f(k)) & = \sum_{k=1}^{j+1} (k-1)^2 f(k) - \sum_{k=0}^j k^2 f(k) \\
    & = \sum_{k=1}^{j+1} k^2 f(k) + \sum_{k=1}^{j+1} f(k) - 2 \sum_{k=1}^{j+1} k f(k) - \sum_{k=0}^j k^2 f(k) \\
    & = (j+1)^2 f(j+1) + \sum_{k=1}^{j+1} f(k) - 2 \sum_{k=1}^{j+1} k f(k). 
\end{align*}
Since $2 \sum_{k=2}^{j+1} k + 1 - j= (j+1)^2$, thus
\begin{align*}
    & (j+1)^2 f(j+1) + \sum_{k=1}^{j+1} f(k) - 2 \sum_{k=1}^{j+1} k f(k) \\
    & = (2 \sum_{k=2}^{j+1} k + 1 - j) f(j+1) + \sum_{k=1}^{j+1} f(k) - 2 \sum_{k=1}^{j+1} k f(k) \\
    & = \sum_{k=2}^{j} (2k -1) (f(j+1) - f(k)) + f(j+1) - f(1).
\end{align*}
As $j \rightarrow \infty$, by Property (A), $f(j+1)$ goes to 1. Moreover, for the sum on the right side above, we have
$$ \sum_{k=2}^{j} (2k -1) (f(j+1) - f(k)) =  \sum_{k=2}^{j} (2k -1) (1 - f(k)) - \sum_{k=2}^{j} (2k -1) (1 - f(j+1)),$$
where, by Property (B),
$$\sum_{k=2}^{j} (2k -1) (1 - f(k)) \ll \sum_{k=2}^{j} \frac{1}{k^{1+\epsilon}} < \infty \quad \text{ as } j \rightarrow \infty, $$
and
$$\sum_{k=2}^{j} (2k -1) (1 - f(j+1)) \ll \sum_{k=2}^{j} \frac{1}{(j+1)^{1+\epsilon}} \ll \frac{1}{(j+1)^\epsilon}$$
which goes to 0 as $j \rightarrow \infty$. Thus, $\sum_{k=2}^{j} (2k -1) (f(j+1) - f(k))$ converges to the constant
$$D_2 := \sum_{k=2}^{\infty} (2k -1) (1 - f(k)) = D_1 + 2 \sum_{k=2}^{\infty} (k -1) (1 - f(k)),$$
and thus combining the above results, we have
$$\mu_{f,2} = 1 - f(1) + D_2.$$
This completes the proof.
\end{proof}

\subsection{Limiting distribution of \texorpdfstring{$M(n)$}{}}

In this subsection, we show that $M(n)$ satisfies an arithmetic distribution as its limiting distribution. We begin by proving the following lemma:

\begin{lma}\label{arith-zeta}
The function
\begin{equation*}
    f_1(n) := \begin{cases}
    0 & \text{ if } n = 0,1, \\
    1/\zeta(n) & \text{ otherwise.}
\end{cases}
\end{equation*}
satisfies Properties (A) and (B) of an arithmetic distribution. In addition, $f_1(n+1) - f(n)$ decreases in $n \geq 2$.
\end{lma}
\begin{proof}
Note that $1/\zeta(n)$ is an increasing function for $n \geq 2$ and $1/\zeta(n) \rightarrow 1$ as $n \rightarrow \infty$ satisfying Property (A). Moreover, the inequality 
$$1 - \frac{1}{\zeta(k)} < \frac{1}{2^{k-1}}$$
for any integer $k \geq 2$ ensures that Property (B) holds as well.
Finally, to prove the last part of the lemma, we show that 
$$\frac{1}{\zeta(n+1)} - \frac{1}{\zeta(n)}$$ 
decreases in $n \geq 2$. To show this, note that $\zeta(s)$ is analytic in $\Re(s) \geq 2$. Thus, defining in this region
$$g(s) = \frac{1}{\zeta(s+1)} - \frac{1}{\zeta(s)},$$
we have
$$g'(s) = - \left( \frac{\zeta'(s+1)}{\zeta(s+1)^2} - \frac{\zeta(s)}{\zeta(s)^2} \right).$$
Note that $\zeta'(s)/\zeta(s)$ is an increasing function in $s$ when $s$ is real. Moreover, $1/\zeta(s)$ is also increasing in $s$ when $s$ is real. Thus, for real $s \geq 2$, we have
$$g'(s) = - \left( \frac{\zeta'(s+1)}{\zeta(s+1)^2} - \frac{\zeta(s)}{\zeta(s)^2} \right) < 0.$$
This establishes that $g(n)$ decreases for $n \geq 2$, and completes the proof.
\end{proof}

As a result, we prove the following limiting distribution for $M(n)$:

\begin{proof}[\textbf{Proof of \thmref{lim-dis-M(n)}}]

For $x >2$ and any integer $k \geq 2$, let $S_k(x)$ denote the number of $k$-free integers less than or equal to $x$. Since there are $S_k(x) - S_{k-1}(x)$ integers $n$ less than or equal to $x$ with $M(n) = k-1$ for any $k \geq 2$, and since (by \cite[(4)]{jala})
$$\lim_{x \rightarrow \infty} \frac{S_k(x)}{x} = \frac{1}{\zeta(k)},$$
thus 
$$\lim_{x \rightarrow \infty} \frac{1}{x} \{ n \leq x : M(n) = 0 \} = \lim_{x \rightarrow \infty} \frac{1}{x} = 0,$$
$$\lim_{x \rightarrow \infty} \frac{1}{x} \{ n \leq x : M(n) = 1 \} = \frac{1}{\zeta(2)},$$
and for $k \geq 2$,
$$\lim_{x \rightarrow \infty} \frac{1}{x} \{ n \leq x : M(n) = k \} = \frac{1}{\zeta(k+1)} - \frac{1}{\zeta(k)}.$$
Thus the limiting distribution of $M(n)$ takes values $k$ with probability $f_1(k+1) - f_1(k)$, where $f_1$ is defined in \eqref{f1n}. This, by definition, proves that $M(n)$ has the arithmetic-$f_1$ distribution as it's limiting distribution.   
\end{proof}

\begin{rmk}
     By \thmref{thm-dis} and the definition of $f_1$ \eqref{f1n}, we find that the first moment of the arithmetic-$f_1$ distribution is 
     $$\mu_{f_1,1} = 1 + \sum_{k=2}^\infty \left( 1 - \frac{1}{\zeta(k)} \right),$$
     and its second moment is
     $$\mu_{f_1,2} = 1 + \sum_{k=2}^\infty \left( 1 - \frac{1}{\zeta(k)} \right) + 2 \sum_{k=2}^\infty (k-1) \left( 1 - \frac{1}{\zeta(k)} \right).$$
     This matches the results for $M(n)$ shown respectively in \eqref{average-M(n)} and \thmref{M(n)-secmom}.
\end{rmk}

\subsection{Other examples}

In this section, we list more applications of \thmref{thm-dis} to other arithmetic functions studied in the literature.

\begin{ex}
For an integer $k \geq 2$, let $\omega_k(n)$ denote the number of distinct prime divisors of $n$ with multiplicity $k$. Let $N_2$ be the set of square-full numbers. For $x > 2$ and an integer $m \geq 0$, let 
$\mathcal{N}_{k,m}(x) = \{ n \leq x : \omega_k(n) = m \}.$
Elma and Martin \cite[Theorem 1]{elmamartin} proved that
$$|\mathcal{N}_{k,m}(x)| = e_{k,m}x + O(x^{1/2} \log x),$$
where
$$e_{k,m} = \frac{6}{\pi^2} \sum_{\substack{l \in N_2 \\ \omega_k(l) = m}} \frac{1}{l} \prod_{p|l} \left( 1 + \frac{1}{p} \right)^{-1},$$
and
$$\sum_{m=0}^\infty e_{k,m} = 1.$$
In fact, in \cite[Section 4]{elmamartin}, the authors explain the process to calculate $e_{k,m}$ explicitly, and in \cite[Table 1]{elmamartin} provide some of these values. In \cite[Corollary 4]{elmamartin}, they showed that, for a fixed $k \geq 2$, $e_{k,m} \leq m^{-(k-o(1))m}.$

Let us define the function
\begin{equation}\label{f3n}
    f_{2,k}(n) := \sum_{m=0}^{n-1} e_{k,m} \quad \text{ if } n \geq 1,  \quad \text{ and } f_{2,k}(0) = 0.
\end{equation}
Clearly, $f_{2,k}(n)$ is non-decreasing and $\lim_{n \rightarrow \infty} f_{2,k}(n) = 1$. Thus, Property (A) holds for $f_{2,k}(n)$. Moreover, since $f_{2,k}(n+1) - f_{2,k}(n) = e_{k,n} \leq n^{-(k-o(1))n}$, thus $f_{2,k}(n+1) - f_{2,k}(n)$ decreases in $n$. 
Finally, we verify that Property (B) holds. Notice that,
$$1 - f_{2,k}(n) = \sum_{m=n}^\infty e_{k,m} \ll  \sum_{m=n}^\infty m^{-(k-o(1))m} \ll n^{-(k-1)n},$$
which makes Condition (B) hold as well. 
This implies that $\omega_k(n)$ satisfies the arithmetic-$f_{2,k}$ distribution as its limiting distribution. Moreover, from the results of Elma and Liu \cite[Theorems 1.1 \& 1.2]{el}, we deduce that
$$\mu_{f_{2,k},1} = \sum_p \frac{p-1}{p^{k+1}},$$
where the sum runs over all primes, and 
$$\mu_{f_{2,k},2} = \left( \sum_p \frac{p-1}{p^{k+1}} \right) \left( \sum_p \frac{p-1}{p^{k+1}} + 1 \right) - \sum_p \frac{1}{p^{2k}} + 2 \sum_p \frac{1}{p^{2k+1}} - \sum_p \frac{1}{p^{2k+2}}.$$
    
\end{ex}


\begin{ex}

In \cite{elmamartin}, the authors generalized the results for $\omega_k(n)$ to a larger set of additive arithmetic functions. Let $A = (a_2,a_3, \ldots)$ be a sequence of complex numbers. Let $\omega_A(n)$ be an additive function defined as
$$\omega_A(n) = \sum_{j=2}^\infty a_j \omega_j(n)$$
which is a finite sum for each $n$. Note that, if $a_j = j-1$, we obtain $\omega_A(n) = \Omega(n) - \omega(n)$, where  $\omega(n)$ count the number of distinct prime factors of a natural number $n$ and $\Omega(n)$ count its total number of prime factors. Moreover, if $a_k= 1$ with $a_j =0$ for all $j \neq k$, we obtain $\omega_A(n) = \omega_k(n)$.

In \cite[Theorem 9]{elmamartin}, they proved that
$$|\{ n \leq x : \omega_A(n) = m \}|= e_{A,m} x + O(x^{1/2} \log x),$$
where
$$e_{A,m} = \frac{6}{\pi^2} \sum_{\substack{l \in N_2 \\ \omega_A(l) = m}} \frac{1}{l} \prod_{p|l} \left( 1 + \frac{1}{p} \right)^{-1}.$$

We assume the function
$$f_A(n) = \sum_{m=0}^{n-1} \left( \frac{6}{\pi^2} \sum_{\substack{l \in N_2 \\ \omega_A(l) = m}} \frac{1}{l} \prod_{p|l} \left( 1 + \frac{1}{p} \right)^{-1} \right) \quad \text{ if } n \geq 1,  \quad \text{ with } f_A(0) = 0.$$

Let $S, E$, and $O$ denote the following sequences:
\begin{align*}
    S & = (1,1,1,1, \ldots), \\
    E & = (1,0,1,0, \ldots), \quad \text{and} \\
    O & = (0, 1, 0, 1, \ldots).
\end{align*}
In \cite[Corollary 14]{elmamartin}, the authors showed that, as $m \rightarrow \infty$, $e_{S,m} \leq m^{-(2-o(1))m}$, $e_{E,m} \leq m^{-(2-o(1))m}$, and $e_{O,m} \leq m^{-(3-o(1))m}$. Thus, following steps in the previous example for $\omega_k(n)$, one can show that $\omega_S(n)$ satisfies an arithmetic-$f_S$ distribution, $\omega_E(n)$ satisfies an arithmetic-$f_E$ distribution, and $\omega_O(n)$ satisfies an arithmetic-$f_O$ distribution.
\end{ex}

In this work, we showed that $m(n)$ is an example of an arithmetic function that has normal order but doesn't satisfy a limiting normal distribution. In the other direction, controlling the variance of an arithmetic function, we can show that limiting normal distribution implies that the function must have a normal order. We present this work in a future article.

\end{document}